\documentclass{amsart}
\usepackage[T1]{fontenc}
\usepackage{graphicx}
\usepackage{hyperref}
\hypersetup{colorlinks=true,linkcolor=Cyan,citecolor=Cyan}
\usepackage[numbers,sort]{natbib}
\usepackage[usenames,dvipsnames]{xcolor}
\usepackage[nameinlink]{cleveref}
\usepackage[color=blue!20!white,textsize=tiny]{todonotes}
\usepackage{tikz-cd}
\usepackage[margin=1.5in]{geometry}
\usepackage[all]{xy}

\newtheorem{theorem}{Theorem}
\newtheorem{prop}[theorem]{Proposition}
\newtheorem{cor}[theorem]{Corollary}
\newtheorem{lemma}[theorem]{Lemma}

\newtheorem{conj}[theorem]{Conjecture}
\newtheorem{question}[theorem]{Question}

\newtheorem{definition}[theorem]{Definition}
\newtheorem*{thm*}{Theorem}

\makeatletter
\newtheorem*{rep@theorem}{\rep@title}
\newcommand{\newreptheorem}[2]{%
\newenvironment{rep#1}[1]{%
 \def\rep@title{#2 \ref{##1}}%
 \begin{rep@theorem}}%
 {\end{rep@theorem}}}
\makeatother

\newreptheorem{theorem}{Theorem}
\newreptheorem{lemma}{Lemma}

\newcommand{\R}{{\ensuremath{\mathbb{R}}}}
\newcommand{\N}{{\ensuremath{\mathbb{N}}}}
\newcommand{\Z}{{\ensuremath{\mathbb{Z}}}}
\newcommand{\Q}{{\ensuremath{\mathbb{Q}}}}

\newcommand{\st[1]}{ST_{#1}}

\title{Satellites of Infinite Rank in the Smooth Concordance Group}
\author{Matthew Hedden}
\address{Department of Mathematics, Michigan State University, East Lansing, MI \ 48823}
\email{mhedden@math.msu.edu}
\author[Juanita Pinz\'on-Caicedo]{Juanita Pinz\'on-Caicedo}
\address {Department of Mathematics, University of Notre Dame, Notre Dame, IN 46556}
\email{jpinzonc@nd.edu}
\date{}

\begin{document}

\begin{abstract}We conjecture that satellite operations are either constant or have infinite rank in the concordance group.  We reduce this to the difficult case of winding number zero satellites, and use $SO(3)$ gauge theory to provide a general criterion sufficient for the image of a satellite operation to generate an infinite rank subgroup of the smooth concordance group $\mathcal{C}$.  Our criterion applies widely; notably to many unknotted patterns for which the corresponding operators on the topological concordance group are zero.  We  raise some questions and conjectures  regarding satellite operators and their interaction with concordance.  \end{abstract}

\maketitle
\section{Introduction}
\thispagestyle{empty}
Oriented knots are said to be {\em concordant} if they cobound a properly embedded cylinder in $[0,1]\times S^3$.  One can vary the regularity of the embeddings, and typically one considers either smooth or locally flat continuous embeddings.  Either choice defines an equivalence relation under which the set of knots becomes an abelian group using the connected sum operation. These concordance groups of knots are intensely studied, with strong motivation provided by the profound distinction between the groups one defines in the topologically locally flat and smooth categories, respectively. Indeed, many questions pertaining to  4--manifolds with small topology (like the $4$--sphere) can be recast or addressed in terms of concordance. Despite the efforts of many mathematicians, the concordance groups are still rather poorly understood. In both categories, for instance, the basic question of whether the groups possess elements of any finite order other than two remains open.  \\

Some of the most powerful tools for analyzing concordance groups come from {\em satellite operations}. To define these, consider a knot in a solid torus $P\subset S^1\times D^2$.   Assign to an arbitrary knot $K$ the image of $P$ under the canonical identification of a neighborhood of $K$ with $S^1\times D^2$. The corresponding {\em satellite knot} is denoted $P(K)$. We refer to the absolute value of the algebraic intersection number of $K$ with $D^2$ as the {\em winding number} of the satellite.  Since framings of a properly embedded annulus in $[0,1]\times S^3$ are naturally identified with framings of either circle on the boundary, it follows that the satellite operator $P$ descends to concordance classes. Thus, for any knot $P\subset S^1\times D^2$ we obtain a self-map on the (smooth or topological) concordance group: 
\[ P:\mathcal{C}\to\mathcal{C}.\]
It is important to note that these maps are typically not homomorphisms. Indeed, the first author has conjectured that they essentially never are:

\begin{conj} \label{homo} The only homomorphisms on the concordance groups induced by satellite operators are the zero map, the identity, and the involution induced by orientation reversal. 
\end{conj}

Despite their conjectural disregard for addition, satellite operations have nonetheless proved to be extremely useful for studying the structure of these groups and figure prominently in many applications and structural theorems, like \cite{cochranteichner,cochranorr,livingston-boundary,fractal,2torsionsolvable,ordertwo,polyone,bipolar,grope,Hom,Levine,injectivity,amphicheiral,reverses,MR780587,davisray}, to cite only a few. \\ 

Particularly noteworthy is work of Cochran-Harvey-Leidy \cite{fractal}, which conjectured a fractal nature of the concordance group derived from the abundance of satellite operators.  As evidence, they introduced the notion of  {\em robust doubling operators} and showed that they interact well with the Cochran-Orr-Teichner  filtration \cite{COT}.    In particular, a robust doubling operator has infinite rank and is injective on large subsets of concordance. More work in this direction came from Cochran-Davis-Ray \cite{injectivity}, who showed that many operators with non-zero winding number are injective on homological variants of the concordance group; Cochran-Harvey-Powell \cite{grope} provided further evidence that concordance is a fractal set by defining  metrics  using gropes, with respect to which winding zero satellites are contractions.  Despite these efforts, it remains unknown  whether   there exists {\em any} injective winding number zero satellite operator. Interesting questions can be made regarding surjectivity as well.  A winding number zero operator is never surjective, as its image consists of knots with bounded genus.  By using an additive concordance invariant with values  bounded by the genus, such as the Ozsv{\'ath}-Szab{\'o}-Stipsicz $\Upsilon$ invariant \cite{OSS}, one can  show that its image cannot even generate concordance; see Wang \cite{Wang} and Livingston \cite{LivingstonUpsilon,LivingstonConcordanceGenus}.    Much more subtle is the winding number one case,  addressed by Levine \cite{Levine}, who showed that there is a winding number one satellite operator whose image does not contain zero. The burgeoning literature on the structure of satellite operations motivates us to make the following conjecture:

\begin{conj}\label{infiniterank} The image of every non-constant satellite operator has infinite rank.
\end{conj}

Since, in light of \Cref{homo}, we have no reason to expect the image to be a subgroup, rank should be interpreted as the rank of the subgroup generated by the image.  It is relatively easy to verify the conjecture in the case of patterns with non-zero winding number, since the algebraic concordance class is additive in an appropriate sense under satellites.  We make this precise in \Cref{non-zero} below. \Cref{infiniterank} then reduces to the winding number zero case, which is significantly harder.  The purpose of this article is to provide a general criterion  to guarantee that such an operator has infinite rank.  To state our result, we recall that the {\em rational linking number} between disjoint curves $\gamma,\eta$ in a rational homology sphere is defined to be $1/d$ times the algebraic intersection number of $\gamma$ with a $2$-chain whose boundary maps to $d\eta$.  Its reduction modulo $\Z$ is the linking form on first homology.   Note that the rational linking number assigns a number to a {\em framed} curve, defined to be the linking of the curve with a push-off defined by the framing. This generalizes the familiar manner in which  framings of  curves in  homology spheres are expressed as integers.  In these terms, we have our main theorem

\begin{theorem}{\label{main}} Let $P\subset S^1\times D^2$ be a pattern with winding number zero, and consider the branched double cover $\Sigma(P(U))$. If $\partial D^2$  has framed lifts to $\Sigma(P(U))$ with non-zero rational linking number, then $P:\mathcal{C}\to\mathcal{C}$ has infinite rank.
\end{theorem}

Here, $\partial D^2$ is equipped with the framing induced by the disk it bounds, and this framing lifts to $\Sigma(P(U))$.  It is interesting to compare our result to that of Cochran-Harvey-Leidy \cite{fractal}.  Robust doubling operators have the property that the Blanchfield self-pairing of a lift of $\partial D^2$ to the infinite cyclic cover of $P(U)$ is non-trivial.  If the homology class of the lift of $\partial D^2$ to $\Sigma(P(U))$ has non-trivial self-pairing under the $\Q/\Z$--valued linking form, it follows easily that its lift to the infinite cyclic cover will have non-trivial Blanchfield self-pairing as well. We do not, however, require any conditions on isotropic submodules of linking forms as in \cite{fractal}. Moreover, our results extend to the case where the branched cover $\Sigma(P(U))$ is a homology sphere (with {\em trivial} linking form). In particular, patterns with $P(U)$ an unknot are of primary interest. In such cases the image of the satellite operator consists of topologically slice knots, so that $P$ acts as zero on the topological concordance group. The methods of \cite{fractal} do not apply in this setting, being manifestly topological. \\

Our result is proved in the context of $SO(3)$ gauge theory, and uses instanton moduli spaces of adapted bundles over 4--manifolds in conjunction with the Chern-Simons invariants of flat connections on the 3--manifolds arising as cross sections of their ends. This technique was pioneered by Furuta \cite{furuta} and Fintushel-Stern \cite{fs-pseudofree}, and later refined by the first author and Kirk \cite{hedden-kirk,hedden-kirk-2} for the purpose of studying such questions. Indeed, our result should be viewed as a vast generalization of the main theorem of \cite{hedden-kirk}, which showed that the Whitehead doubling operator has infinite rank, and of \cite{pinzon}, which showed that some generalizations of the Whitehead doubling operator also have infinite rank. The technique  is inherently smooth; indeed, as remarked above, any unknotted pattern will be zero on topological concordance. It is also worth emphasizing that it seems very difficult, if not impossible, to prove a result of this form using most other known smooth concordance invariants. For instance, $\Upsilon$ cannot prove that an infinite set of knots whose genera are bounded (like the image of any of our patterns) are independent in concordance \cite[Theorem 9.2]{LivingstonUpsilon}; it seems similarly unlikely that any invariants derived from the stable equivalence class of the knot Floer homology complex can prove such a result \cite{Hom2}. While the correction terms of branched covers \cite{ManolescuOwens,Jabuka} or knot Floer homology of the branch loci therein \cite{GRS} should contain a great deal of concordance information, they are quite challenging to compute. Similarly, the various concordance invariants coming from Khovanov homology and its generalizations \cite{Rasmussen,Lobb,LewarkLobb}  are extremely difficult to compute for families and are not expected to behave predictably under satellites.\\

We use the instanton cobordism obstruction to show that, given a pattern satisfying the rational linking number hypothesis, an infinite collection of torus knots can be chosen such that their images under $P$ are $\Z$--linearly independent.   While we use an independent set of knots (a subset of torus knots) to prove our theorem, we suspect that this was unnecessary. In fact, we make a rather bold strengthening of our conjecture:

\begin{conj} For any non-constant winding number zero operator $P$, there exists a knot $K$ for which the set $\{P(nK)\}_{n\in\Z}$ has infinite rank.
\end{conj}  

Thus we expect satellite operators to expand the concordance group, in the sense that the image of a finite rank subgroup will often have infinite rank. This again drives home  the expectation that they are not homomorphisms. It would be very interesting to verify that Whitehead doubles of the $n$--fold connected sums of the trefoil knot are independent in concordance when $n>0$, or to find {\em any} knot for which the Whitehead double of both $K$ and $-K$ are non-zero in concordance. \\

Finally, we remark that in $3$--dimensional topology, JSJ theory applied to knot complements tells us that the decomposition of a knot as an iterated satellite is quite rigid.  It is reasonable to ask whether there are  $4$--dimensional remnants of this rigidity.  As a sample, given a winding number zero satellite operator $P$, we can define the {\em $P$--filtration} of  the concordance group  to be the descending filtration  whose $i$--th term is the subgroup generated by $i$--fold iterated satellites with pattern $P$:
\[  ...\subseteq \langle P^2(\mathcal{C})\rangle \subseteq \langle P(\mathcal{C})\rangle \subset \mathcal{C}.  \]
                  
 \begin{question}\label{it-filt} If $P$ is a non-constant winding zero operator, does each associated graded group of the $P$--filtration have infinite rank?
 \end{question}
 
Since the genera of knots in the image of $P$ are bounded, Shida Wang's results imply the first quotient $\mathcal{C}/\langle P(\mathcal{C})\rangle$ is of infinite rank \cite{Wang}. Again, it would be interesting to understand even the seemingly simpler question of whether iterations of a pattern $P$ on subsets of $\mathcal{C}$ are independent. In the case of iterated Whitehead doubles, Kyungbae Park \cite{park} has shown that the first two iterated Whitehead doubles of the trefoil are independent. For patterns with non-zero winding number, Wenzhao Chen gave a criterion for iterates of an operator to be independent and has shown that infinitely many iterated satellites with the Mazur pattern are independent \cite{ChenThesis}.  As a final question, we have
 
 \begin{question} If $P$ and $Q$ are winding number zero satellite operators inducing isomorphic filtrations of $\mathcal{C}$, are $P$ and $Q$ concordant as knots in the solid torus?
 \end{question}

\textbf{Outline:} The next section verifies \Cref{infiniterank} for satellites with non-zero winding number, then turns to topological constructions essential for the proof of  our main result.  \Cref{sec:instanton} offers an overview of the instanton cobordism obstruction derived from instanton moduli spaces and Chern-Simons invariants. \Cref{sec:proof} uses this obstruction in conjunction with the topological constructions from \Cref{sec:topology} to prove  the main result. Finally, \Cref{sec:examples}  contains some examples that illustrate the way our result can be applied in practice.  In particular, we show how to check whether an operator satisfies the hypothesis of \Cref{main} and how to use this method to easily produce examples of topologically slice operators with infinite rank on smooth concordance. \\

\textbf{Acknowledgements:} The  authors owe a debt of gratitude to Paul Kirk, with whom the first author learned and explored the ideas and questions central to the paper, and whose mentorship and guidance showed the second author the way forward from her time as a graduate student.  The second author is eternally grateful to Tye Lidman for many very fruitful conversations, and to him, Allison Miller, and Chuck Livingston for many helpful comments on an early draft.  Matthew Hedden was partially supported by NSF CAREER grant DMS-1150872, NSF grant DMS-1709016 and an Alfred P. Sloan Research Fellowship. Juanita Pinz{\'o}n-Caicedo was partially supported by NSF grant DMS-1664567, she is also grateful to the Max Planck Institute for Mathematics in Bonn for its hospitality and financial support while a portion of this work was prepared for publication.

\section{Topological Preliminaries and Key Constructions}\label{sec:topology}
 
In this section we introduce notation and then verify  \Cref{infiniterank} in the case of satellite operators with non-zero winding number.  The proof breaks down completely when the winding number is zero, but it inspires the rough idea that one should look for invariants of winding number zero satellite knots that remember the companion and are robust enough to distinguish infinitely generated subgroups of the topologically slice subgroup of the smooth concordance group.   For us, these will be gauge theoretic properties and invariants of the 2-fold branched covers which behave well with respect to definite cobordisms.  \\

To begin, we provide a more precise definition of a satellite. 

\begin{definition} \label{def:satellite}
Let $P\subset S^1\times D^2$ be an oriented knot in the solid torus. Consider an orientation-preserving embedding $h:S^1\times D^2\to S^3$ whose image is a tubular neighborhood of a knot $K$ so that  $S^1\times \{*\in \partial D^2\}$ is mapped to the canonical longitude of $K$. The knot $h(P)$ is called the {\em satellite knot with pattern $P$} and {\em companion $K$}, and is  denoted $P(K)$. The {\em winding number} of the satellite is defined to be the algebraic intersection number of $P$ with  $\{*\}\times D^2$.  \end{definition}
 
We now verify \Cref{infiniterank} for patterns with  non-zero winding number.  

\begin{prop}\label{non-zero} Let $P\subset S^1\times D^2$ be a pattern for a satellite operator with non-zero winding number, $w$. Then $P:\mathcal{C}\to\mathcal{C}$ has infinite rank.
\end{prop}
\begin{proof}  

We appeal to the jump function of the Tristram-Levine signature, whose definition and basic properties we quickly review.  Given  a knot $K$, let $V_K$ be a Seifert matrix and consider the associated matrix $$A(t)=(1-t)V_K + (1-t^{-1})V_K^{T}=(1-t)(V_K-t^{-1}V_K^{T}).$$ If $\zeta=  e^{{2\pi i}\theta} \in S^1$ is a unit complex number, then $A(\zeta)$ is a Hermitian matrix. For those $\zeta\in S^1$ for which $A(\zeta)$ is non-singular, we let $\sigma_K(\zeta)$ denote the signature of $A(\zeta)$, and extend this definition to all of $S^1$ by taking the average of the one-sided limits. Since $\det(A(t))=(1-t)^{2g}\Delta_K(t^{-1})$, where $\Delta_K(t)$ is the Alexander polynomial of $K$, the function $\sigma_K:S^1\to\Z$ is continuous except at the unit roots of $\Delta_K(t)$, and is therefore a step function with finitely many jumps. These discontinuities are recorded by the {\em signature jump function}, defined as 
\[\delta_K(\zeta) =\frac{1}{2}\left( \lim_{s \to \zeta^+}\sigma_K(s) -  \lim_{s \to \zeta^-}\sigma_K(s)\right).\]

We now observe that, while the signature function  is not a concordance invariant, if $K_1$ and $K_2$ are concordant, then $\sigma_{K_1}(\zeta)=\sigma_{K_2}(\zeta)$ for those $\zeta\in S^1$  that are not a root of $\Delta_{K_1}$ or $\Delta_{K_2}$.  This is an immediate consequence of the fact that the Seifert form of a slice knot, in this case $K_1\#m({K_2}^r)$, is metabolic (here $m(K)$ denotes the mirror image).  It  follows that the jump function {\em is} a concordance invariant:  $\delta_{K_1}(\zeta)=\delta_{K_2}(\zeta)$ for all $\zeta\in S^1$ if $K_1\sim K_2$. Moreover, the jump function provides a homomorphism from the (topological) concordance group to $\Z$. More details can be found in \cite[Section 13]{gordon-survey}. Finally, notice that the jump function $\delta_K$ is zero except at the roots of $\Delta_K$ and  therefore the support of the jump functions of a finite set of knots is finite. Thus, to show that $P$ has infinite rank it is enough to show that there exists a family of knots $\{K_i\}_{i\in \N}$ for which the set $\{\zeta\in S^1 \mid \delta_{P(K_i)}(\zeta)\ne 0\}$ is infinite. \\

Litherland proved a  formula for the signature function of a satellite knot \cite{litherland}. Expressed in terms of the jump function, this formula reads as
\begin{equation}\label{satellitejump}\delta_{P(K)}(\zeta)= \delta_{P(U)}(\zeta)+\delta_K(\zeta^w),\end{equation}
where $w$ is the winding number. Litherland also showed that the set $\left\{\delta_{T_{p,q}}\right\}$ of jump functions of torus knots is independent in the additive group of functions on the circle \cite{litherland}. In particular, the set $\left\{\zeta\in S^1 \mid \delta_{T_{p,q}}(\zeta)\ne 0 \text{ for some } T_{p,q}\right\}$ contains  infinitely many distinct elements. Since the function $\zeta\to \zeta^w$ is finite to one, the set $\left\{\zeta\in S^1 \mid \delta_{T_{p,q}}(\zeta^w)\ne 0 \text{ for some } T_{p,q}\right\}$ is also infinite.  It then follows from  \eqref{satellitejump} that the support of jump functions of satellites of torus knots with pattern $P$ is infinite as well. Thus, the operator defined by $P$ has infinite rank.

\end{proof}

The proof of the above proposition clearly breaks down in the winding number zero case, as the signature function of the satellite forgets the companion knot entirely.  This is not simply a deficit of the signature function.  Indeed, in the case that $P$ is unknotted, or even merely an Alexander polynomial one knot,  the topological concordance class of $P(K)$ is trivial for {\em any} $K$ by work of Freedman \cite{Freedman}. Thus, far subtler techniques are required.  For the remainder of this section, we pave our way to apply an instanton cobordism obstruction to the problem at hand by constructing some explicit cobordisms.  \\

Recall, then, that  closed oriented 3--manifolds $Y_0$ and $Y_1$ are {\em oriented cobordant} if there exists a compact oriented 4--manifold $W$ with oriented boundary $\partial W= -Y_0\sqcup Y_1$, where we follow the ``outward normal first''  convention for orientations induced on the boundary. The manifold $W$ is called a {\em cobordism from $Y_0$ to $Y_1$}, and $Y_0$ (resp. $Y_1$) will be referred to as  the ``incoming'' (resp. ``outgoing") boundary component.  We will construct cobordisms whose incoming boundary component is the branched double cover of a satellite knot, and  whose outgoing boundary  contains a  manifold obtained by Dehn surgery on the companion as a summand. These cobordisms allow us to isolate the companion knot $K$ from $P(K)$, and thereby utilize desirable gauge theoretic properties of  surgeries on the former to obstruct sliceness of the latter in the spirit of the signature argument.  \\

The $2$--fold cover of a 3-manifold  branched over a knot $K$ will be denoted by $\Sigma\left(K\right)$, as the 3-manifold should be obvious from the context. In the case of a satellite knot $P(K)$ in the $3$-sphere, there is a decomposition of $\Sigma\left(P(K)\right)$ as the union of the 2-fold cover of $S^1\times D^2$ branched over $P$ and a 2-fold covering space of the knot complement $S^3\setminus N(K)$; see, for example \cite{seifert} or \cite{livingston-melvin}. The isomorphism type of this latter covering space depends only on the parity of the winding number.  In the case that the winding number is even, the cover is the disjoint union of two copies of $S^3\setminus N(K)$.  We make this decomposition precise with the following proposition. An image depicting the decomposition can be found in \Cref{twist}.

\begin{prop}\label{decomposition} Let $P\subset S^1\times D^2$ be a pattern knot with even winding number, and let $\Sigma(P)$ denote the 2-fold cover of the solid torus, branched over $P$.  Then there is a decomposition of the 2-fold branched cover of the satellite $\Sigma(P(K))$:
\begin{equation}\label{eq:decomposition} \Sigma\left(P(K)\right)\cong \Sigma(P) \underset{\widetilde{h}}{\cup} 2\left(S^3\setminus N(K)\right),\end{equation} where the gluing takes place along $\partial \Sigma(P)=T_1\sqcup T_2$, $T_i=\partial\left(S^3\setminus N(K)\right)_i$. Moreover, the gluing map $\widetilde{h}$ identifies the lift of $(*\times \partial D^2, S^1\times *)$ in $T_i$ with the pair $(\mu_K,\lambda_K)$ in the corresponding copy of $S^3\setminus N(K)$.
\end{prop}

\begin{proof} The branched double cover of  $S^1\times D^2$ branched over $P$ is, away from $P$, simply the 2-fold cover defined by the homomorphism \[H_1\left(S^1\times D^2\setminus P\right)\cong \Z\oplus \Z\rightarrow \Z\rightarrow \Z/2\] that measures the parity of the projection onto the summand generated by the meridian of $P$.   As the pattern knot has even winding number, the restriction of this homomorphism to the subgroup of $\pi_1\left(S^1\times D^2\setminus P\right)$ generated by the boundary torus is trivial.  Thus, the restriction of the branched covering to the boundary of the solid torus is the trivial 2-fold cover, consisting of the disjoint union of two 2-tori $T_1$ and $T_2$, interchanged by the covering involution.  It follows that we can extend the branched covering involution on $\Sigma(P)$ to an involution on the manifold $\Sigma(P) {\cup} 2\left(S^3\setminus N(K)\right)$, where  the boundary tori of $\Sigma(P)$ and the knot complements are identified as stated.  The quotient of this latter space under the involution is the $3$-sphere, obtained from the solid torus and the knot complement  by identifying their boundaries with a diffeomorphism $h$ satisfying $h(*\times \partial D^2)=\mu_K$ and $h(S^1\times *)=\lambda_K$.  The image of the branch locus $P\subset S^1\times D^2$  under this identification is the satellite $P(K)$, hence the manifold described by the right side of  \Cref{eq:decomposition} is homeomorphic to the 2-fold branched covering of the 3-sphere branched over $P(K)$, as claimed.
\end{proof}

Next, we briefly recall the definition of the {\em rational linking number} of a pair of disjoint oriented closed curves in a rational homology sphere, $\Sigma$ (see Seifert and Threlfall \cite[Section 77, pg 288]{ST} for a good reference).  Suppose $\gamma,\eta$ are two such curves.  Alexander duality over $\Q$ implies $H_1(\Sigma\setminus \gamma)\cong \Q$, and under this isomorphism $[\eta]\in H_1(\Sigma\setminus\gamma)$ can be expressed as a  multiple of the class of the meridian to $\gamma$.   The rational linking number  $lk_\Sigma(\gamma,\eta)$ is defined to be this multiple:  \[[\eta]= lk_\Sigma(\gamma,\eta) \cdot [\mu_\gamma] \]  

A more geometric perspective is provided by intersection numbers with {\em rational Seifert surfaces}. A rational Seifert surface for an oriented curve $\gamma\subset \Sigma$ is a smoothly embedded surface $F_\gamma \subset \Sigma\setminus N(\gamma)$  representing a generator of $H_2(\Sigma\setminus N(\gamma),\partial )\cong\Z$,   whose boundary is a positive multiple of $\gamma$ in homology: $[\partial F_\gamma]=d\cdot [\gamma]\in H_1(N(\gamma))$, with $d>0$.  The rational linking number  of $\gamma$ and $\eta$ can alternatively be characterized as follows:

\[ lk_\Sigma(\gamma,\eta):=  \frac{1}{d}( F_\gamma \cdot \eta) \in \Q,\] 
 where $\cdot$ denotes algebraic intersection number.  The rational linking number is symmetric in $\gamma$ and $\eta$, and can be viewed as a simultaneous generalization of the ordinary linking number between null-homologous curves and the linking pairing on first homology.  Indeed, the value of  $lk_\Sigma(\gamma,\eta)$ modulo $\Z$ depends only on the homology classes  $[\gamma],[\eta]$ and equals the  $\Q/\Z$--valued linking pairing on $H_1(\Sigma)$.  Like ordinary linking numbers, the rational linking number is an invariant of the concordance class of the link $\gamma\sqcup\eta$.  Finally, we observe that the linking number assigns a well-defined rational number to a {\em framed} curve, defined to be the linking number between a curve and its push-off using the given framing. \\

A framed curve $\eta$ in a 3--manifold $\Sigma$ also determines a smooth 4--manifold $W_\eta$ given by 2--handle attachment along $\eta$.  The incoming and outgoing boundary components of $W_\eta$ are $\Sigma$ and $\Sigma_\eta$, respectively, where $\Sigma_\eta$ is the 3--manifold obtained by Dehn surgery along $\eta$ using its given framing. The following lemma describes the constraints on the topology of the 2-handle cobordism imposed by the rational linking number. Homology groups will be assumed to have coefficients in $\Z$.

\begin{lemma} \label{lemma:intersection} Let $\Sigma$ be a rational homology sphere, and $\eta\subset \Sigma$ a framed curve whose homology class $[\eta]\in H_1(\Sigma)$ has order $d>0$.  Then the 2-handle cobordism $W_\eta$ has $H_2(W_\eta)\cong \Z$, and the self-intersection of any surface representing a generator is equal to $d^2\cdot lk_\Sigma(\eta,\eta')$, where $\eta'$ is a framed push-off.   The orders of first homology of the boundaries are related by the linking number, provided it is non-zero:
\[ |H_1(\Sigma_\eta)| = |lk_\Sigma(\eta,\eta')| \cdot |H_1(\Sigma)|\]
\end{lemma}
Note, in particular, that the sign of the self-intersection of any class in $H_2(W_\eta)$  agrees with the sign of the framing, viewed as an element of $\Q$ via the rational linking number.  Also observe that the special case of null-homologous knots $d=1$ is well-known.

\begin{proof}  Lefschetz duality gives rise to a unimodular pairing:
\[   H_2(W_\eta,\Sigma)\otimes H_2(W_\eta,\Sigma_\eta)\rightarrow  \Z,\]
which one can calculate geometrically via intersections between properly embedded surfaces.  The relevant groups, $H_2(W_\eta,\Sigma)\cong\Z$ and $H_2(W_\eta,\Sigma_\eta)\cong\Z$, are generated by the classes of the core $D^2\times *$  and cocore $*\times D^2$ of the 2-handle $D^2\times D^2$, respectively, which we denote $[C]$ and $[C^*]$. They are oriented in the standard way, so that they intersect once, positively. \\

As $\Sigma$ is a rational homology sphere, the long exact sequence of the pair $(W_\eta,\Sigma)$ shows that $H_2(W_\eta)\cong \Z$. There is an intersection pairing on this latter group, and it remains to compute the self-pairing of a generator.  To construct a generator, consider a rational Seifert surface $F_\eta$ for $\eta$.  By definition, $\partial F_\eta$ is an embedded multi-curve in $\partial N(\eta)$, and is  homologous to $d\cdot [\eta'] - k\cdot[\mu]$, where $\eta'$ is the framing of $\eta$, $\mu$ is the meridian, and $d,k\in \Z$, with $d>0$. Let $G$ be a singular 2-chain in $\partial N(\eta)$  realizing a homology between $\partial F_\eta$ and $d$ parallel copies of $\eta'$ union $k$ oppositely oriented copies of $\mu$.  We can construct a closed $2$-chain $\widehat{F_\eta}$ representing a generator of $H_2(W_\eta)$ by considering the union of $F_\eta$, $G$,  $d$ parallel copies of the core of the 2-handle, and $k$ copies of the cocore:   $\widehat{F}_\eta=F_\eta \cup G \cup dC\cup kC^*$.  \\    

The  pairings above are natural with respect to the maps on homology induced by the inclusions $W_\eta \subset (W_\eta,\Sigma)$ and $W_\eta \subset (W_\eta,\Sigma_\eta)$, in the sense that there is a commutative triangle
\[
\xymatrix{
H_2(W_\eta)\otimes H_2(W_\eta) \ar[d] \ar[dr]& \\
       H_2(W_\eta,\Sigma)\otimes H_2(W_\eta,\Sigma_\eta) \ar[r] &  \Z ,
}\]
where the diagonal and  horizontal maps are the  pairings on absolute and relative homology, respectively.  Now $[\widehat{F_\eta}]$ maps to  $d[C]$ and $k[C^*]$ in $H_2(W_\eta,\Sigma)$ and $H_2(W_\eta,\Sigma_\eta)$, respectively, under the inclusions. Indeed, this holds on the chain level since $F_\eta\cup G\cup kC^*$ and $F_\eta\cup G\cup dC$ are represented by  chains in $C_2(\Sigma)$ and $C_2(\Sigma_\eta)$, respectively. Naturality of the pairings  then implies  \[[\widehat{F_\eta}]\cdot[ \widehat{F_\eta}]=d[C]\cdot k[C^*]=dk.\]   On the other hand, the framing of $\eta$ is identified with the rational number \[ lk_{\Sigma}(\eta,\eta')=\frac{1}{d} (F_\eta\cdot \eta')=\frac{k}{d},\]  where the last equality follows from the fact that $[\partial F_\eta]=d\cdot [\eta'] - k\cdot[\mu]$.  Thus the self pairing of $[\widehat{F}_\eta]$ is equal to $d^2 \cdot lk_{\Sigma}(\eta,\eta')$, as claimed.\\

We analyze $H_1(\Sigma_\eta)$  using the exact sequence of the pair $(W_\eta,\Sigma_\eta)$:
\[
\xymatrix{
0 \ar[r] & H_2(\Sigma_\eta) \ar[r] &  H_2(W_\eta)\ar[r] \ar[d]^{\rotatebox[origin=c]{270}{$\cong$}} &  H_2(W_\eta,\Sigma_\eta)\ar[r] \ar[d]^{\rotatebox[origin=c]{270}{$\cong$}} &  H_1(\Sigma_\eta)\ar[r] &  H_1(W_\eta)\ar[r] &  0\\
& & \Z \ar[r]^{\times k}& \Z & & & 
}
\]

where $k=d\cdot lk_{\Sigma}(\eta,\eta')$ as above, and we use that $H_3(W_\eta,\Sigma_\eta)=H_1(W_\eta,\Sigma_\eta)=0$ (which follows from the fact that $W_\eta$ can also be viewed as $2$-handle cobordism from $\Sigma_\eta$ to $\Sigma$). If $k\ne 0$,  the sequence reduces to
\[
\xymatrix{
0 \ar[r] &\Z \ar[r]^{\times k} & \Z \ar[r] & H_1(\Sigma_\eta) \ar[r] & H_1(W) \ar[r]&  0,
}
\]
which leads, in turn, to an exact sequence
\[
\xymatrix{
0 \ar[r] &\Z/k \ar[r] & H_1(\Sigma_\eta) \ar[r] & H_1(W) \ar[r]&  0.
}
\]
But $H_1(W)$ is the quotient of $H_1(\Sigma)$  by the subgroup generated by $[\eta]$, a subgroup of order $d$. It follows that  \[ |H_1(\Sigma_\eta)|= |k|\cdot \frac{|H_1(\Sigma)|}{d} = |lk_{\Sigma}(\eta,\eta')| \cdot |H_1(\Sigma)|.\]
\end{proof}
 
The following proposition provides the cobordisms requisite for our strategy. We remind the reader that for any choice of companion $K$, the branched cover $\Sigma(P(K))$ contains a copy of $\Sigma(P)$, the 2-fold cover of the solid torus $S^1\times D^2$ branched over the pattern $P$. Each lift of $\partial D^2$ is contained in $\partial\Sigma(P)$ where it is identified with $\mu_K\subset \partial \left(S^3\setminus N(K)\right)$  as in \Cref{decomposition}.  The curve $\partial D^2$ acquires the (canonical) Seifert framing from the disk that it bounds, and this framing is equivalent (up to isotopy of sections of the unit normal bundle) to the framing which $\partial D^2$ receives by virtue of being embedded in the oriented surface $\partial (S^1\times D^2)$.  The lifting criterion for covering spaces allows us to lift the framing to $\Sigma(P)$, where, as above, we conflate the framing with a curve in the neighborhood of  $\partial D^2$ via the tubular neighborhood theorem. It follows that the lifted framing of $\partial D^2$ is independent of $K$. In particular, $lk_{\Sigma(P(U))}(\mu_U,\mu'_U)=lk_{\Sigma(P(K))}(\mu_K,\mu'_K)$.

\begin{prop}\label{cobordism} If $P\subset S^1\times D^2$ is a pattern with winding number zero and $(p,q)$ is a pair of arbitrary relatively prime nonzero integers, then $\Sigma(P(K))$ is cobordant to a 3--manifold containing $S^3_{p/q}(K)$ as a  connected summand.  Moreover, if the rational linking number of (either) framed lift of $\partial D^2$ to $\Sigma(P(U))$ satisfies $lk_{\Sigma(P(U))}(\mu_U,\mu'_U)<-q/p$, then the cobordism can be taken to be negative definite.
\end{prop}

\begin{proof} Referring to our decomposition of the cover $\Sigma(P(K))$ in \Cref{decomposition}, let $T=\partial \left(S^3\setminus N(K)\right)\subset \Sigma(P(K))$ denote the boundary torus of either copy of the knot complement found in $\Sigma(P(K))$.    The cobordism posited by the proposition is obtained by attaching a 4--dimensional 2-handle to a curve $\eta\subset T$, with framing $\eta'$ given by a parallel copy of $\eta$ in the torus.  There are two steps. First, we  show that the boundary of handle attachment along $\eta$ with torus framing $\eta'$  contains  Dehn surgery along $K$ as a connected summand.  This is a well-known result, but we include a proof for completeness. We then determine the sign of the intersection form for this 2-handle cobordism in terms of the rational linking number of the framed lift of $\partial D^2$ and the slope of $\eta$. \\

To  clarify the computations, endow $H_1(T)$ with a basis represented by a meridian-longitude pair $(\mu_K,\lambda_K)$ for the knot $K$, viewed as a subset of $S^3$.  Thus a slope $\eta\subset T$ for $p/q$ Dehn filling is given by a simple closed curve in $T$ representing the element  $p\cdot[\mu_K]+q\cdot[\lambda_K]\in H_1(T)$.   A tubular neighborhood of $\eta$ intersects the torus $T$ in a submanifold $A\subset T$ which is homeomorphic to an annulus $[0,1]\times S^1$. Let $\eta'$ be the framing of $\eta$ given by either boundary component of this annulus.   \\

Form a smooth 4--manifold $W$ by attaching a 4-dimensional 2-handle $D^2\times D^2$ to the product $[0,1]\times \Sigma(P(K))$ along $\{1\}\times \eta$ with framing given by $\{1\}\times\eta'$. To see that the ``outgoing'' component of $\partial W$ is a connected sum $S^3_{p/q}(K)\# Y$, recall that the 2-handle attachment replaces a tubular neighborhood $N(\eta)$ of $\{1\}\times\eta$ in $\{1\}\times\Sigma(P(K))$ with the solid torus $D^2\times \partial D^2\subset D^2\times D^2$. Next, decompose the belt sphere $\{*\}\times \partial D^2$ into two arcs and consider the corresponding decomposition of $D^2\times \partial D^2$ into the union of two 3-dimensional 2-handles $H_+$ and $H_-$. Identify $N(\eta)$ as a product $[-1,1]\times A$, where $[-1,0]\times A$ represents the intersection of $N(\eta)$ with $S^3\setminus N(K)$, and $[0,1]\times A$ represents the intersection of $N(\eta)$ with $\Sigma(P(K))\setminus \large{(}S^3\setminus N(K)\large{)}$. Notice that removing $[0,1]\times A$ from $\Sigma(P(K))\setminus \large{(}S^3\setminus N(K)\large{)}$ does not change its homeomorphism type, nor does removing $[-1,0]\times A$ from $S^3\setminus N(K)$. The complement of $N(\eta)$ in $\{1\}\times\Sigma(P(K))$ can therefore be written as a union of \begin{equation}\label{eq:etacomplement} \Sigma(P(K))\setminus \large{(}S^3\setminus N(K)\large{)}\qquad\text{and}\qquad S^3\setminus N(K),\end{equation} glued along the annulus $T\setminus A$. Therefore the ``outgoing'' component of $\partial W$ is obtained as the union of the pieces $$\Sigma(P(K))\setminus \large{(}S^3\setminus N(K)\large{)} \cup H_+ \qquad\text{and}\qquad S^3\setminus N(K) \cup H_-,$$ with the gluing of the 3-dimensional 2-handle determined  in each case by the identification of a central curve of $\{\pm 1\}\times A$ with the central circle of $H_\pm$. In other words, these 3-manifolds are precisely $$Y\setminus B^3 \qquad\text{and}\qquad S^3_{p/q}(K)\setminus B^3,$$ where $Y$ denotes Dehn filling of $\Sigma(P(K))\setminus \large{(}S^3\setminus N(K)\large{)}$. This shows that the ``outgoing'' component of $\partial W$ is diffeomorphic to $Y\#S^3_{p/q}(K).$ For more details see \cite[Lemma 7.2]{Gordon} or \cite[Section 2.2]{GordonLectures}.\\

To  complete the proof, we analyze the intersection form of the 2-handle cobordism $W$.  According to \Cref{lemma:intersection},  $H_2(W)\cong \Z$, and the sign of the intersection form is given by the sign of the rational linking number of the curve $\eta$ with its framing $\eta'$.  One can compute $lk_{\Sigma(P(K))}(\eta,\eta')$ from its  definition as the rational homology class represented by one curve in the complement of the other. To begin, we claim that \[[\eta']=p\cdot[\mu'_K]\in H_1(\Sigma(P(K))\setminus N(\eta)),\] where $\mu'_K$ denotes a parallel  copy of $\mu_K$  in the interior of the knot complement $S^3\setminus N(K)$, embedded in  $\Sigma(P(K))\setminus N(\eta)$ according to the decomposition \eqref{eq:etacomplement} of the latter.  Here, ``parallel'' is defined by the framing lifted from the Seifert framing of $\partial D^2$.  Since  $[\eta]=p\cdot[\mu_K]+q\cdot [\lambda_K]\in H_1(T)$, and $[\lambda_K]=0\in H_1(S^3\setminus N(K))$, the claimed equality follows.  Therefore, \[  lk_{\Sigma(P(K))}(\eta,\eta')= p\cdot lk_{\Sigma(P(K))}(\eta, \mu_K').\]
Now observe that the equality $[\eta]=p\cdot [\mu_K]+ q\cdot [\lambda_K]\in H_1(T)$ also holds in the homology of the complement of $\mu_K'$, since the splitting torus lies in the complement of  $\mu_K'$.  Therefore
\[ lk_{\Sigma(P(K))}(\eta, \mu_K') = p\cdot lk_{\Sigma(P(K))}(\mu_K, \mu_K') + q \cdot lk_{\Sigma(P(K))}(\lambda_K, \mu_K').\]
The linking of $\lambda_K$ and $\mu_K'$  equals one, by considering the intersection of a Seifert surface for $K$ with $\mu_K'$.  Combining the two equalities, we arrive at
\[ lk_{\Sigma(P(K))}(\eta,\eta') = p^2\cdot lk_{\Sigma(P(K))}(\mu_K, \mu_K') + pq.\]  
Thus the linking number, and hence the intersection pairing on $H_2(W_\eta)$, will be negative if and only if $p^2lk_{\Sigma(P(K))}(\mu_K,\mu_K')+ pq<0$ or, equivalently, if and only if $lk_{\Sigma(P(K))}(\mu_K,\mu_K')=lk_{\Sigma(P(U))}(\mu_U,\mu'_U)<-q/p$, as claimed.
\end{proof}

We derive a corollary from \Cref{cobordism} that will be used to prove \Cref{main}.

\begin{cor}\label{cob-torus}Let $P\subset S^1\times D^2$ be a pattern with winding number zero, and let $K$ be a knot which can be unknotted by changing only positive crossings. If $lk_{\Sigma(P(U))}(\mu_U,\mu'_U)<0$, then there exists relatively prime positive integers $p,q$ and a negative definite cobordism $W$ from $\Sigma(P(K))$ to $S^3_{p/q}(K)\#Y$, where both $S^3_{p/q}(K)$ and $Y$ are $\Z/2$--homology spheres, and $Y$ depends only on $P$ and the pair $(p,q)$.\end{cor}

Knots which can be unknotted by changing positive crossings are abundant, take positive knots for example.  For our purposes, the positive torus knots $T_{r,s}$ will be sufficient. 

\begin{proof} Choose $\eta$ to be a curve in $T=\partial(S^3\setminus N(K))\subset \Sigma(P(K))$ representing the element $p\cdot[\mu_K]+q\cdot[\lambda_K]\in H_1(T;\Z)$ as in the preceding proposition. Assume $q/p$ lies in the interval $(0,-lk_{\Sigma(P(K))}(\mu_K,\mu'_K))$, where  $lk_{\Sigma(P(K))}(\mu_K,\mu'_K)=lk_{\Sigma(P(U))}(\mu_U,\mu'_U)$  denotes the rational linking number in $\Sigma(P(K))$ between the meridian and its parallel copy determined by the framing lifted  from $\partial D^2$, using the decomposition of  \Cref{decomposition}. The cobordism $W$ described in the proof of \Cref{cobordism} has intersection form presented by  $d^2(p^2lk_\Sigma(\mu_K,\mu'_K))+pq)$, which is negative since $0<q/p<-lk_{\Sigma(P(K))}(\mu_K,\mu'_K)$.    According to \Cref{lemma:intersection}, the order of the first homology of the outgoing boundary of $W$ is given by
\[ | lk_{\Sigma(P(K)}(\eta,\eta')|\cdot |H_1(\Sigma(P(K))| = |p^2lk_{\Sigma(P(K))}(\mu_K,\mu'_K)+pq|\cdot |H_1(\Sigma(P(K))|.\]
The order of $H_1(\Sigma(P(K))$ is odd, since the branched 2-fold cover of a knot in $S^3$ is a $\Z/2$-homology sphere.  Picking $p$ to be odd, and the parity of $q$ opposite that of the numerator of  $lk_{\Sigma(P(K))}(\mu_K,\mu'_K)$, ensures that the product above is also odd. Hence, the outgoing boundary component is a $\Z/2$-homology sphere.\\

Now the outgoing end of the cobordism $W$ contains $S^3_{p/q}(K)$ as a connected summand, and its remaining summand is the 3-manifold resulting from Dehn filling the complement in $\Sigma(P(K))$ of one of the lifts of $S^3\setminus N(K)$.  In terms of the decomposition of \Cref{decomposition}, this is a Dehn filling of the manifold 
\[  \Sigma(P) \underset{\widetilde{h}}{\cup} \left(S^3\setminus N(K)\right). \]
Since $K$ can be unknotted by a series of positive-to-negative crossing changes, there is a negative definite cobordism (rel. boundary) from $S^3\setminus N(K)$ to $S^3\setminus N(U)\cong S^1\times D^2$,  see \cite[Lemma 3.5]{hedden-kirk} or \cite{cochran-gompf}. Applying this relative cobordism to the remaining copy of $S^3\setminus N(K)$ in the outgoing boundary component of $W$, we obtain a new cobordism whose outgoing manifold is the connected sum of $S^3_{p/q}(K)$ with a manifold $Y$. Notice that $Y$ is obtained by filling both boundary components of $\Sigma(P)$ with solid tori, and thus it depends only on the pattern $P$ and the integers $(p,q)$. Moreover,  $Y\# S^3_{p/q}(K)$ has homology isomorphic to that of the outgoing end of $W$, since the latter is obtained by a sequence of $-1$ surgeries along a null-homologous split link of unknots (which realize the crossing changes in an unknotting sequence for $K$ as in \cite[Chapter 6.C]{Rolfsen}). Choosing $p,q$ as above ensures $Y\# S^3_{p/q}(K)$, and hence $Y$ is a $\Z/2$-homology sphere.\end{proof}

\section{Instanton Obstruction to Sliceness}\label{sec:instanton}
In this section we survey a method which uses moduli spaces of anti-self-dual (ASD) connections on $SO(3)$ bundles over $4$--manifolds with cylindrical ends to study the $3$--dimensional $\Z/2$--homology cobordism group.  This technique provides an obstruction to the existence of a negative definite $4$--manifold whose boundary is a given disjoint union of $3$--manifolds.  To explain it,   recall that the relative Chern-Simons invariant $cs(\alpha,\beta)$ between  flat $SO(3)$ connections $\alpha,\beta$ on a closed oriented $3$--manifold, is  defined as the integral \[ cs(\alpha,\beta):=  -\frac{1}{8\pi}\int_{[0,1]\times Y} Tr(F(A_t)\wedge F(A_t))\in \mathbb{R}/\mathbb{Z},\]
where $A_t$ is any path of connections  between $\alpha$ and $\beta$.  The right hand side is the Chern-Weil integrand for the first Pontryagin class.   Integrality of the first Pontryagin number of a bundle on a closed $4$--manifold implies that, modulo $\Z$, the integral is independent of the chosen path, and invariant under the gauge group actions on $\alpha$ and $\beta$. The obstruction is phrased in terms of the {\em minimal Chern-Simons invariant}.  For a $\Z/2$--homology 3--sphere $Y$, this is defined as \[\tau(Y):=\min\{cs(\alpha,\theta) \mid \alpha \text{ flat connection on }Y\}\in (0,1],\] where $\theta$ is the trivial connection on the unique (trivial) $SO(3)$ bundle on $Y$, and where we have identified $\R/\Z$ with $(0,1]$ in the obvious way (we could, in fact, lift the relative and minimal Chern-Simons values to $\R/4\Z$, but will have no need to do so for our purposes. See Section 2.2 of \cite{hedden-kirk-2} for more details.)  \\

Using Fintushel-Stern's results on equivariant Yang-Mills theory on pseudofree orbifolds \cite{fs-pseudofree}  in conjunction with the Chern-Simons invariants, Furuta developed a powerful cobordism obstruction \cite[Theorem 2.1]{furuta} based on compactness results for equivariant ASD connections, see also  \cite[Theorem 5.1]{fs-inst}. Drawing on work of Floer \cite{Floer}, Taubes \cite{Taubes-periodic,Taubes-L2}, Morgan-Mrowka-Ruberman \cite{MMR}, and Donaldson \cite{Donaldson},  Hedden and Kirk extended Furuta's technique to non-compact $4$-manifolds with more general cylindrical ends.  The following theorem, in the case that the numerator $p$ of the surgery coefficient  is $1$ and  with the additional assumption that the obstructed $4$-manifold $X$ is an integral homology punctured sphere, is a restatement of Furuta's theorem \cite[Theorem 2.1]{furuta}. Furuta mentions, without proof, that the result extends to obstruct  $X$ with definite intersection form and with ``certain torsions in their homology groups.'' Such extensions were taken up and further refined by Hedden-Kirk \cite{hedden-kirk-2}, and the $p>1$ case stated here utilizes their refinements for definite $4$-manifolds whose ends have $\Z/2$-homology sphere cross sections.  We will use \cite{hedden-kirk-2} as the main reference since it provides a convenient unified treatment. The following result will be used to establish \Cref{main}.

\begin{theorem}[Furuta \cite{furuta} $p=1$, Hedden-Kirk \cite{hedden-kirk-2} $p>1$]\label{cobound} Consider a family $\left\{\Sigma_i\right\}_{i=1}^{N}$ of oriented $\Z/2$--homology 3-spheres. Let $(p,q)$ and $(r,s)$ be two pairs of relatively prime and positive integers and suppose $\Sigma_N=S^3_{p/q}(T_{r,s})$.  If
\begin{equation}\label{criterion}
\frac{p}{rs(qrs-p)}<\min\left\{\tfrac{1}{r},\;\tfrac{1}{s},\;\tfrac{1}{qrs-p},\;\tau(\pm\Sigma_1),\ldots,\tau(\pm\Sigma_{N-1})\right\},\end{equation}
then there does not exist a smooth  $4$--manifold $X$ with $H^1(X,\Z/2)=0$ and negative definite intersection form, whose oriented boundary is given by  
\[ \partial X= \coprod\limits_{i=1}^{N} a_i \Sigma_i,\ \ \mathrm{with}\ \ a_i\in\Z, a_N> 0.\] 
\end{theorem}

In the above, $a_i\Sigma_i$ means the disjoint union of $a_i$ copies of $\Sigma_i$, endowed with the given orientation if $a_i>0$ and opposite otherwise. We  sketch a proof of the theorem, argued by way of contradiction.
\\

{\em Sketch of proof:}  The manifold $-\Sigma_N=-S^3_{p/q}(T_{r,s})$ is  diffeomorphic to  the link of the complex surface singularity \[ z_0^r+z_1^s+z_2^{qrs-p}=0.\] With this description, $-\Sigma_N$ is  clearly Seifert fibered, and bounds two closely related negative definite smooth $4$--manifolds.  The first is a resolution $R$ of the singularity by repeated blow-ups.  The second is the  smooth negative definite $4$--manifold $W$ obtained  from the mapping cylinder of the Seifert fibration $-\Sigma_N\overset{\pi}\to S^2$ by excising neighborhoods of the singularities that arise from the singular fibers.  Thus $W$ has three additional lens space boundary components, $L(r,b_0)$, $L(s,b_1)$ and $L(qrs-p,b_2)$, where we suppress the precise values of  $b_i$ for simplicity. Over  $W$ one constructs a non-trivial $SO(3)$ bundle $E$ by stabilizing an $SO(2)$ bundle $\mathcal{L}$ whose Euler class $e(\mathcal{L})\in H^2(W)\cong \Z$ generates. One then defines an associated moduli space $\mathcal{M}$ of ASD connections on $E$. To do so, we must prescribe boundary conditions for the connections in $\mathcal{M}$, which amounts to specifying flat connections on the bundles that arise by restricting $E$ to the components of $\partial W$, see \cite[Definition 2.1 and Section 2.3]{hedden-kirk-2}. We take the trivial flat connection on  $-\Sigma_N$, and non-trivial flat connections on the lens space boundary components which are determined by the unique flat connections for the restriction of $\mathcal{L}$ along them, see  \cite[Lemmas 2.10, 2.11]{hedden-kirk-2}.  One then attaches cylindrical ends to $W$ and defines $\mathcal{M}$ to be gauge equivalence classes of ASD connections which limit exponentially fast in an appropriate Sobolev space to the prescribed flat connections along the ends.   Assuming non-degeneracy of the boundary flat connections, the virtual dimension of this moduli space can be computed using the Atiyah-Patodi-Singer index theorem \cite[Proposition 2.6]{hedden-kirk-2}. Flat connections on lens spaces are non-degenerate, and the trivial flat connection on {\em any} rational homology sphere is non-degenerate, hence the index theorem computes the virtual dimension in terms of rho invariants of the lens spaces.  Using the Neumann-Zagier formula \cite{nz}, we find that the dimension equals $1$ whenever $(p,q)$ and $(r,s)$ are pairs of relatively prime positive integers, see \cite[Equation 3.4]{hedden-kirk-2}.  Moreover, the moduli space is non-empty, since it possesses a unique singular point that corresponds,  by way of Hodge theory, to an explicit reducible connection \cite[Lemma 2.12]{hedden-kirk-2}. The slice theorem \cite[Theorem 4.13]{Donaldson} endows the singular point with a neighborhood  diffeomorphic to a half open interval $[0,\epsilon)$. \\

Assume that a negative definite $4$--manifold $X$ as in the theorem exists. One can glue it to the disjoint union of $W$ and $(a_N-1)$ copies of the resolution $R$ along  $a_N\Sigma_N$ to obtain \[\widehat{X}=X\cup W\cup (a_N-1)R.\]  The fact that $X$ and $R$ are negative definite allows one to argue, again through the Atiyah-Patodi-Singer index theorem, that the extension $\widehat{E}$ of  $E\to W$ by the trivial bundles over $X$ and $R$ possesses a  moduli space $\widehat{\mathcal{M}}$ of ASD connections with the same virtual dimension as $\mathcal{M}$.  The flat connections along the new ends arising from the  boundary components of $X$ are all taken to be trivial, and we rely here on the fact that the trivial flat connection  is non-degenerate.     As above, Hodge theory shows that reducible connections on $\widehat{E}$ give rise to   singular points of $\widehat{\mathcal{M}}$.  The enumeration of these singularities is more subtle for $\widehat{E}$, but the  hypothesis that the boundary components of $X$ are $\Z/2$-homology spheres reduces it, by \cite[Theorem 2.16]{ hedden-kirk-2},  to the  calculation of a quantity derived from the intersection form on $H^2(X)$, see \cite[Definition 2.14]{hedden-kirk-2}.  Our  assumption that $H^1(X;\Z/2)=0$  implies that the order of the torsion subgroup of $H^2(X,\partial X)$ is odd, which further implies that the aforementioned cohomological quantity is also odd, see \cite[Proof of Proposition 2.1]{hedden-kirk}.  It follows that the  number of singular points in $\widehat{\mathcal{M}}$ is odd.  This shows that the moduli space $\widehat{\mathcal{M}}$ is  non-compact, since a compact $1$--manifold has an even number of boundary points (which we identify with the singular points using their $[0,\epsilon)$ neighborhoods).\\

Now failure of compactness in moduli spaces of ASD connections on manifolds with cylindrical ends occurs only through  bubbling \cite{Uhlenbeck} or by energy escaping down the ends in the form of broken  flow lines for the gradient of the Chern-Simons functional \cite{Floer,Taubes-periodic,MMR,Donaldson}.  Each of these phenomena require a certain quanta of  analytic energy which, for an ASD connection, is given by the integral over $\widehat{X}$ of the integrand defining the Chern-Simons invariant; for points in the moduli space $\widehat{\mathcal{M}}$, the energy is given by $q/rs(qrs-p)$, a quantity determined by the cup-square of the Euler class of $\mathcal{L}$.  The assumptions  in  \eqref{criterion} guarantee that bubbling cannot occur (since this quantity is less than $4$), and that a sequence cannot diverge to a broken flow line. Indeed, the minimal Chern-Simons invariants of $\pm\Sigma_i$ and the lens space boundary components $L(r,b_0)$, $L(s,b_1)$, $L(qrs-p,b_2)$ provide a lower bound for the amount of energy carried by a broken flow line. The first three terms in \eqref{criterion} give lower bounds for the minimal Chern-Simons invariant of these lens spaces. It follows that  the moduli space {\em is} compact, a contradiction which rules out the existence of $X$.  See \cite[Section 2.5]{hedden-kirk-2} for more details on the compactness result for $\widehat{\mathcal{M}}$, which relies crucially on Morgan-Mrowka-Ruberman's ``convergence with no loss of energy'' theorem \cite[Theorem 6.3.3]{MMR}. \qed \\

The following corollary explains how negative definite cobordisms can be used to study linear independence  in the $\Z/2$--homology cobordism group.  Recall that the $\Z/2$--homology cobordism group $\Theta_{\Z/2}$ consists of equivalence classes of oriented $\Z/2$--homology $3$--spheres, where two such are equivalent if they cobound a homology cylinder.  Addition is given by connected sum.   

\begin{cor}\label{cobound-sum} Let $\left\{\Sigma_i\right\}_{i=1}^{N}$ be a family of oriented $\Z/2$--homology 3-spheres. Suppose $\Sigma_N$ is  cobordant via a negative definite cobordism with $H^1(Z;\Z/2)=0$ to $S^3_{p/q}(T_{r,s})\#Y$, where $Y$ is any $\Z/2$--homology 3--sphere and $(p,q)$ and $(r,s)$ are pairs of relatively prime  positive integers with $p$ odd. If
\begin{equation}
\frac{p}{rs(qrs-p)}<\min\left\{\tfrac{1}{r},\;\tfrac{1}{s},\;\tfrac{1}{qrs-p},\;\tau(\pm\Sigma_1),\ldots,\tau(\pm\Sigma_{N-1}),\tau(Y)\right\},\end{equation} then $\Sigma_N$ has infinite order in $\Theta_{\Z/2}$ and is independent from the other manifolds: $$\langle\Sigma_N\rangle\cap \langle\Sigma_1,\ldots,\Sigma_{N-1}\rangle =\{0\},$$ where $\langle-\rangle$ denotes the subgroup of $\Theta_{\Z/2}$ generated by  $-$.
\end{cor}

\begin{proof} Suppose, to the contrary, that either the intersection of the subgroups in question is non-trivial or $\Sigma_N$ has finite order in $\Theta_{\Z/2}$.  This implies  there exist integers $c_1,\ldots,c_N$ with $c_N>0$ such that $$c_N\Sigma_N\#\overline{\left(c_1\Sigma_1\#\ldots\#c_{N-1}\Sigma_{N-1}\right)}=\partial Q,$$ for $Q$ a smooth 4--manifold with the same $\Z/2$--homology groups as the 4--ball. Here, we temporarily use the notation $c_i\Sigma_i$ to denote the connected sum (instead of disjoint union) of $c_i$ copies of $\Sigma_i$, with opposite its given orientation if $c_i<0$.  Form the 4--manifold \[{X_0}=Q\underset{c_N\Sigma_N}{\cup} c_N Z,\] where $c_N Z$ is the boundary connected sum of $c_N$ copies of  the negative definite cobordism from $\Sigma_N$ to $S^3_{p/q}(T_{r,s})\#Y$. Attaching $3$--handles to $X_0$ along all of the connected sum spheres yields a smooth negative definite 4--manifold $X$ with $H^1(X;\Z/2)=0$, and boundary  
\[ \partial X = c_N S^3_{p/q}(T_{r,s})\coprod c_N Y \coprod\limits_{i=1}^{N-1} -c_i \Sigma_i,\ \ \mathrm{with}\ \ c_i\in\Z, c_N> 0.\] 

Relabelling as appropriate, we arrive at a contradiction to \Cref{cobound}.
\end{proof}

 There is a well-known homomorphism from the concordance group to the $\Z/2$--homology cobordism group
\[ \Sigma: \mathcal{C}\to \Theta_{\Z/2},\]
defined by sending the concordance class of a knot $K$ to the homology cobordism class of its branched $2$--fold cover $\Sigma(K)$.  This homomorphism,  in conjunction with \Cref{cobound}, provides a  tool  for showing that the image of a satellite operator has infinite rank.  Indeed, to show that an operator $P$ has infinite rank, we only need to argue that the composition $\Sigma \circ P$ does.  For this, it suffices to find an infinite  collection of torus knots $\{T_{r_i,s_i}\}_{i=1}^\infty$ for which the branched covers of their satellites $\{\Sigma(P(T_{r_i,s_i}))\}_{i=1}^\infty$ are independent in $\Theta_{\Z/2}$.  We accomplish this in the next section.

\section{Proof of the Main Result}\label{sec:proof}
We are now in a position to use the instanton  obstruction from the previous section in conjunction with the cobordisms  constructed in  \Cref{sec:topology} to prove our main theorem.  For convenience, we restate it here. 

\begin{reptheorem}{main} Let $P\subset S^1\times D^2$ be a pattern with winding number zero, and consider the branched double cover $\Sigma(P(U))$. If $\partial D^2$, equipped with the Seifert framing from $D^2$, has framed lifts in $\Sigma(P(U))$ with non-zero rational linking number, then there exists an infinite family of knots $\{K_i\}_{i=1}^\infty$ for which $\{P(K_i)\}_{i=1}^\infty$ is  a $\Z$--independent family in $\mathcal{C}$. 
\end{reptheorem}

The family $\{K_i\}_{i=1}^\infty$ will be a carefully selected subset of the torus knots, chosen so that the instanton obstruction  can be applied to the collection $\{\Sigma(P(K_i))\}_{i=1}^\infty$ to establish its independence in the $\Z/2$--homology cobordism group $\Theta_{\Z/2}$.  By the remarks at the end of the last section, this will  show that $\{P(K_i)\}_{i=1}^\infty$ is an infinite  independent family in $\mathcal{C}$. We  construct the family of torus knots recursively.  Note that the set of all torus knots is linearly independent in $\mathcal{C}$, by  Litherland's result  \cite{litherland}, though this fact will not be necessary for our proof.

\begin{proof} Recall that, according to \Cref{decomposition}, the lift of $\partial D^2$ to $\Sigma(P)$ gets identified with $\mu_K$ in $\Sigma(P(K))$ for any choice of companion $K$. Denote by \[l=lk_{\Sigma(P(U))}(\mu_U,\mu'_U)=lk_{\Sigma(P(K))}(\mu_K,\mu'_K) \] the rational linking number between the lifts of $\partial D^2$ and its Seifert framed push-off. Assume first that $l$ is strictly negative.  Choose, once and for all, a pair of relatively prime positive integers $p,q$ with $p$ odd, $q$ having parity opposite that of the numerator of  $l$, and satisfying $l<-q/p<0$.   In terms of these, we define the function of two variables
\[ e(r,s)=\frac{p}{rs(qrs-p)}.\]
  \Cref{cob-torus} provides a negative definite cobordism from $\Sigma(P(T_{r,s}))$ to $S^3_{p/q}(T_{r,s})\#Y$ for any choice of  torus knot.  We therefore choose a pair of relatively prime positive integers $r_1,s_1$ so that  $$e(r_1,s_1)<\min\left\{\tau(Y),\;\;\tfrac{1}{r_1},\;\tfrac{1}{s_1},\;\tfrac{1}{qr_1s_1-p}\right\}.$$

Letting $K_1=T_{r_1,s_1}$, and $\Sigma_1=\Sigma(P(K_1))$, the hypothesis for \Cref{cobound-sum} are met, thereby showing that $\Sigma_1$ has infinite order in $\Theta_{\Z/2}$, in other words, $\langle \Sigma_1\rangle \cong\Z$.\\

Now choose a pair of relatively prime positive integers $r_2,s_2$ so that $$e(r_2,s_2)<\min\left\{\tau(Y),\;\tau(\pm\Sigma_1),\;\tfrac{1}{r_2},\;\tfrac{1}{s_2},\;\tfrac{1}{qr_2s_2-p}\right\}.$$ 

 \Cref{cob-torus} again gives a negative definite cobordism, now from $\Sigma_2=\Sigma(P(T_{r_2,s_2}))$ to $S^3_{p/q}(T_{r_2,s_2})\#Y$.  \Cref{cobound-sum}, together with our previous choices, gives $\langle \Sigma_2\rangle \cong \Z$, and $ \langle \Sigma_2\rangle \cap \langle \Sigma_1\rangle=0$, so that $\langle \Sigma_2,\Sigma_1\rangle\cong \Z^2$.\\

In general, suppose knots $K_1,\ldots,K_{N-1}$ have been chosen so that \begin{equation}\label{sum}\langle \Sigma_1,\Sigma_2,\ldots,\Sigma_{N-1}\rangle \cong \Z^{N-1}.\end{equation} Proceeding as before, choose $r_N,s_N$ so that 
 $$e(r_N,s_N)<\min\left\{\tau(Y),\;\tau(\pm\Sigma_{1}),\; \ldots,\;\tau(\pm\Sigma_{N-1}),\;\tfrac{1}{r_N},\;\tfrac{1}{s_N},\;\tfrac{1}{qr_Ns_N-p}\right\}.$$ 
For $K_N=T_{r_N,s_N}$ and $\Sigma_N=\Sigma(P(K_N))$, \Cref{cobound-sum} together with \Cref{sum} shows that $\langle \Sigma_1,\Sigma_2,\ldots,\Sigma_{N}\rangle \cong \Z^{N}$.\\

This recursive procedure  defines a family $\{K_i\}_{i=1}^\infty$. Any finite subset  of $\{P(K_i)\}_{i=1}^\infty$  generates a full rank subgroup in $\mathcal{C}$, since we have shown its image under  the homomorphism $\Sigma:\mathcal{C}\rightarrow \Theta_{\Z/2}$ has full rank.  Since the finite subset can be chosen arbitrarily, $\{P(K_i)\}_{i=1}^\infty$ is an independent family in $\mathcal{C}$. \\

Having treated the case when  $l=lk_{\Sigma(P(U))}(\mu_U,\mu'_U)$ is negative, we notice that if this quantity is positive for a pattern $P$ then its mirror $m(P)$ has rational linking number negative. Thus, if $\{m(P)(K_i)\}_{i=1}^\infty$ is an independent family, then $\{P(m(K_i))\}_{i=1}^\infty$ is independent as well.
\end{proof}

\section{Examples}\label{sec:examples}

In this final section we present a method to compute $l$, the lift of the Seifert framing of $\partial D^2$ to $\Sigma(P(U))$. We also verify that some explicit satellite operators have infinite rank, and indicate a robust method for obtaining new examples.

\subsection{Computing the rational linking number}
The hypothesis required of a winding number zero pattern by our theorem can be easily verified in terms of linear algebra, as we now explain. The algorithm we present follows from  a simple combination of a formula of Cha and Ko \cite[Theorem 3.1]{cha-ko},  which computes linking numbers in rational homology spheres described as integral surgery on a framed link, with an algorithm described by Akbulut and Kirby in \cite{akbulut-kirby} to realize branched covers as surgery on a link derived from a Seifert surface for the branching set.\\

We first obtain a link surgery presentation for $\Sigma(P(U))$.  To construct this, begin with a  Seifert surface $F\subset S^1\times D^2$ for the pattern $P$ (which exists, by the winding number zero assumption). An unknotted embedding of $S^1\times D^2$  into the $3$-sphere  allows us to consider $F$ as a Seifert surface for $P(U)$. A framed link $L$ describing $\Sigma(P(U))$ in terms of surgery is  obtained from $F$ by the algorithm in \cite[Section 2]{akbulut-kirby}.  If  $V$ is a Seifert matrix for $F$, then the linking matrix for this surgery presentation is given by \[A= [V+V^T].\]See  \Cref{twist} for an example of this procedure.   \\ 

Now, for an arbitrary curve $\gamma$ embedded in the complement of $L$, denote by $[lk_{S^3}(\gamma,L)]$ the row vector whose $i$--th entry is $lk(\gamma,L_i)$, where $L_i$ is the $i$--th component of $L$.  The formula of  Cha and Ko \cite[Theorem 3.1]{cha-ko} expresses the linking of curves in $\Sigma(P(U))$ as follows:
\[ lk_{\Sigma(P(U))}(\eta,\gamma) = lk_{S^3}(\eta,\gamma)-[lk_{S^3} (\eta,  L)]\cdot [A]^{-1}\cdot [lk_{S^3} (\gamma,  L)]^T. \]
Let $J_1,J_2$ denote the framed lifts of $\partial D^2$ to the surgery presentation of $\Sigma(P(U))$.   We wish to compute $l= lk_{\Sigma(P(U))}(J_i,J_i')$.  The linking number of $J_i$ with $J_i'$, viewed  as curves in $S^3$, will be zero.  This is because the framing is the lift of the Seifert framing.   Thus, in this case the first term in the  formula above vanishes,  yielding  
\begin{equation}\label{lk-sigma} lk_{\Sigma(P(U))}(J_i,J_i') =-[lk_{S^3} (J_i, L)]\cdot[V+V^T]^{-1}\cdot[lk_{S^3} (J_i', L)]^T.\end{equation}

We can compute this quantity concretely, without reference to the surgery presentation.  If  $\{a_i,b_i\}_{i=1}^g$ denotes a basis for $H_1(F;\Z)\cong \Z^{2g}$, we have an Alexander dual basis $a_i^*,b_i^*$ for $H_1(S^3\setminus F;\Z)\cong \Z^{2g}$,  given as follows:  for a curve $e\subset F\subset S^3$ in the basis, its dual $e^*\subset S^3\setminus F$ is a curve linking $e$ exactly once, and linking no other basis curve. To check whether the pattern knot satisfies our hypothesis, one expresses $\partial D^2$ in terms of the Alexander dual basis, by a vector we denote $lk(\partial D^2)$. Then \eqref{lk-sigma} shows that the quantity \[  l=  -lk(\partial D^2) [V+V^T]^{-1} lk(\partial D^2)^T\] is exactly $lk_{\Sigma(P(U))}(J_i,J_i')$, the linking number of a framed lift of $\partial D^2$ to the branched double cover. 

\subsection{Method for producing Satellite Operators of Infinite Rank}
The previous calculation of the linking number   yields  a concrete method for producing infinite rank winding number zero satellite operators. To do this, consider {\em any} embedded surface $F\subset S^3$ with  a single boundary component.  The boundary $\partial F$ will be a (possibly trivial) knot, $P$.   Now consider any {\em unknotted} curve $\gamma\subset S^3\setminus F$.  Then the complement of $\gamma$ is homeomorphic to a solid torus, with $P$ embedded therein with winding number zero (since $F$ is a Seifert surface for $P$ in the complement).   One can express $\gamma$, as above, as a vector in terms of the Alexander dual basis for $H_1(S^3\setminus F;\Z)\cong \Z^{2g}$, and compute the linking number of its framed lift to the branched double cover of $P$ by the same formula.  This provides a far-reaching method for producing satellite operators with infinite rank.  We can also easily specify embeddings of surfaces whose Seifert forms have trivial Alexander polynomial (for instance, by taking higher genus Seifert surfaces for an unknot, as in the case of the Whitehead doubling operator).  Finding appropriate $\gamma$ in these cases will produce satellite operators with infinite rank on smooth concordance, but which represent the zero map on topological concordance. \\

\subsubsection{Genus one satellite operators with trivial Alexander polynomial} 
To illustrate this last point, let $F$ have genus one and let $\{a,b\}$ be the basis for $H_1(F;\Z)$. Imposing the condition that the Alexander polynomial of $\partial F$ be trivial implies that in terms of the given basis the Seifert form for $F$ has matrix   \[V=\left[\begin{array}{cc}n & m \\m-1 & k\end{array}\right],\ \text{where} \ nk=m(m-1).\]  Then $$lk_\Sigma(J,J')=(x,y)\left[\begin{array}{cc}2k & 1-2m \\1-2m & 2n\end{array}\right]\left(\begin{array}{c}x \\y\end{array}\right), $$ where $x=lk(\gamma,a)$, and $y=lk(\gamma,b)$. It follows that $lk_\Sigma(J,J')\neq 0$ as long as $(x,y)$ is not an integral multiple of the elements $$\begin{cases}(1,0)\text{ or }\left(\tfrac{n}{2m-1},1\right) & \text{ if } k=0\\
(0,1)\text{ or }\left(1,\tfrac{k}{2m-1}\right) & \text{ if } n=0\\
\left(\tfrac{m}{gcd(m,k)},\tfrac{k}{gcd(m,k)}\right)\text{ or }\left(\tfrac{m-1}{gcd(m-1,k)},\tfrac{k}{gcd(m-1,k)}\right)& \text{ if } n,k\neq 0
 \end{cases}.$$
 Any unknotted curve  $\gamma\subset S^3\setminus F$ whose homology class lies in the complement of the above vectors will produce a satellite operator with infinite rank on smooth concordance which is topologically trivial.  
\Cref{genus1} gives a specific example.
\begin{figure}[ht]
\centering
\begin{minipage}{0.5\textwidth}
        \centering
         \def\svgwidth{0.5\textwidth}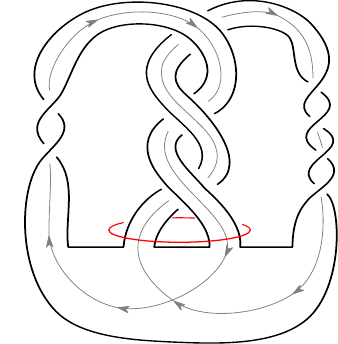
    \end{minipage}%
    \begin{minipage}{0.5\textwidth}
    The Seifert matrix for this surface with basis curves shown is $$V=\left[\begin{array}{cc}1 & 2 \\1 & 2\end{array}\right].$$ Since $(x,y)=(-1,1)$ is not an integer multiple of $(2,1)$ or $(1,1)$, the pattern defined by this particular choice of axis has infinite rank as a map on concordance.
    \end{minipage}
\caption{A genus 1 satellite operator acting as zero on topological concordance.}\label{genus1}
\end{figure}

\subsubsection{Twisted Whitehead doubles}
Let $P_k$ be the twist knot with $k$ full twists, and consider the curve $\gamma$  shown in \Cref{twist}. Viewing $P_k$ as a knot in the solid torus $S^3\setminus N(\gamma)$ defines a pattern for a satellite operator, the {\em positive clasped $k$-twisted Whitehead doubling operator}. Notice that  $P_0\subset S^3\setminus N(\gamma)$ defines the pattern for the positive untwisted Whitehead double. The Seifert surface  $F_{k}$ specified by \Cref{twist} has Seifert matrix \[V_{k}=\left[\begin{array}{cc}k & 0 \\-1 & -1\end{array}\right].\] The algorithm described in \cite{akbulut-kirby} to produce a surgery description for $\Sigma(P_k(U))$ yields the framed link $L$ shown to the right of \Cref{twist}, and it follows from the algorithm itself that the linking vector of either $J_i$ (the lifts of $\gamma$) with $L$ is precisely the linking vector of $\gamma$ with the cores of the 1-handles of $F_{k}$. Thus $\vec{lk}(J_1,L)=(1,0)$ and $l=lk_\Sigma(J_1,J_1')=-\frac{2}{-4k-1}\ne 0$ so that $P_k\subset S^3\setminus \gamma$ is of infinite rank for all $k$.  \\

\begin{figure}[ht]
    \centering
    \begin{tabular}{p{0.5\textwidth}cp{0.5\textwidth}}
    \def\svgwidth{0.3\textwidth}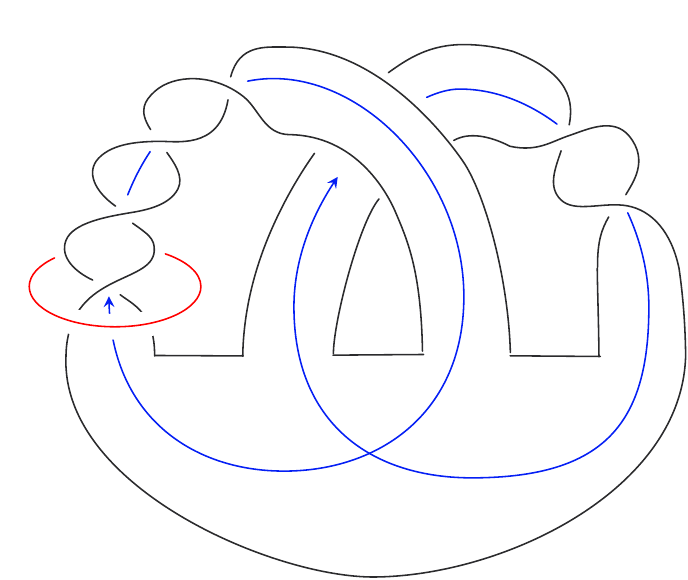&\hspace{0.1\textwidth}&
    \def\svgwidth{0.3\textwidth} 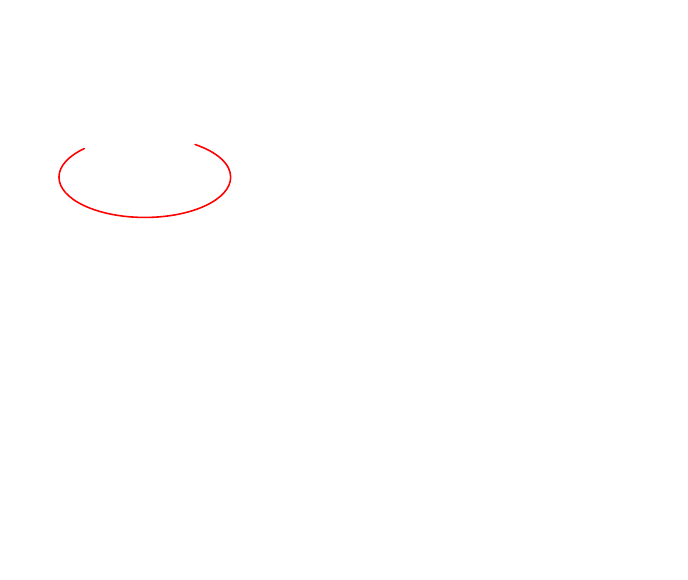\\
   A Seifert surface for the twist knot $P_k$, embedded in $S^3\setminus N(\gamma)\cong S^1\times D^2$. Shown is the case $k=2$. && Kirby diagram for $\Sigma(P_k(U))$ obtained from the Seifert surface for $P_k$. The complement of the lifts of $\gamma=\partial D^2$ (shown in red) is homeomorphic to $\Sigma(P_k)$, and $\Sigma(P_k(K))$ is obtained by replacing neighborhoods of $J_i$ with $S^3\setminus N(K)$.
    \end{tabular}
    \caption{The twisted whitehead doubling operator $P_k$ }\label{twist}
\end{figure}

\subsection{A few words about the case $l=0$}
While the above discussion shows that patterns with $l\neq 0$ are abundant and easily produced,  there are certainly  patterns with $l=0$. For example, if the link $P\sqcup \gamma$ defining an embedding $P\subset S^1\times D^2$ is split, then the operator $P:\mathcal{C}\to \mathcal{C}$ is constant with value $[P(U)]$. Similarly, if $P$ is slice in the complement of $\gamma$, then $P:\mathcal{C}\to \mathcal{C}$ is the zero map.  Such examples show that our result is in some sense optimal; without the linking number condition, there are examples of patterns whose associated maps have rank zero. \\

An interesting class of winding number zero patterns with $l=0$ is provided by those $P\sqcup \gamma$ which are {\em boundary links}, links whose components bound disjoint Seifert surfaces.  If $P\sqcup \gamma$ is a boundary link, then the lift of the Seifert framing of $\gamma=\partial D^2$ to $\Sigma(P(U))$ will clearly be zero.  Indeed, the Seifert surface for $\gamma$ in the complement of the Seifert surface for $P$ will lift to $\Sigma(P(U))$.  Thus a lift of $\gamma$ is null-homologous, and the Seifert framing lifts to the Seifert framing.  It would be interesting to find refined conditions that could address the structure of  satellite operators  produced in this realm.  \\

Of particular interest are iterated operators.  Let $P^r= P\circ P\circ \ldots \circ P$ denote the $r$--th iteration of a winding number zero pattern $P$, and let $P^r\sqcup \gamma_r$ be the 2-component link which defines the relevant embedding of the unknot in the solid torus. More precisely, denote by  $\st[r]$ the solid torus containing $P^r$, and $h_{r}:\st[r]\to N(P)\subset \st[1]$ the homeomorphism that defines $P^{r+1}$. Since the winding number of $P \subset S^1\times D^2$ is zero, there exists  a Seifert surface $S$ for $\gamma_1$ in $\st[1]$ that is disjoint from the first iterate $P$, and $F_r$ a Seifert surface for $P^r$ that lies in the interior of $\st[r]$. Then $h_{r}(F_{r})$ is a Seifert surface for $P^{r+1}$ contained in $N(P)$ and thus disjoint from $S$. This shows that $P^{r+1}\sqcup \gamma$ is a boundary link, and hence any iterated satellite operator of winding number zero will have $l=0$. This point seems to indicate that the technique utilized by \Cref{main}, while quite useful, is a first level obstruction with respect to the $P$-filtration of the concordance group defined in the introduction.  It would be quite interesting to develop tools  sensitive to 4-dimensional aspects of the JSJ decomposition which could analyze the higher terms in these filtrations.  \\

\bibliographystyle{alpha}
\bibliography{references}
\end{document}